\documentclass[11pt,reqno]{amsart}
\usepackage[utf8]{inputenc}
\usepackage{amsopn,amssymb,mathrsfs,xcolor,url,enumitem,accents,marginnote}
\usepackage[abbrev]{amsrefs} 
\usepackage[a4paper,lmargin=3cm,rmargin=3cm,tmargin=4cm,bmargin=4cm,marginparwidth=2.8cm,marginparsep=1mm]{geometry} 
\usepackage[bookmarks=true,hyperindex,pdftex,colorlinks,citecolor=blue]{hyperref} 
\usepackage{cleveref}

\makeatletter
\@namedef{subjclassname@2020}{\textup{2020} Mathematics Subject Classification}
\makeatother

\newtheorem{theorem}{Theorem}[section]

\newtheorem{lemma}[theorem]{Lemma}
\newtheorem{proposition}[theorem]{Proposition}
\newtheorem{problem}[theorem]{Problem}

\theoremstyle{definition}

\newtheorem{remark}[theorem]{Remark}

\newcommand{\cl}[1]{\ensuremath{\overline{{#1}}}}
\newcommand{\da}[2]{\ensuremath{\langle{#1},{#2}\rangle}}
\newcommand{\DEL}[2]{\ensuremath{\mathbf{\Delta}^{#1}_{#2}}}
\newcommand{\ep}{\varepsilon}
\newcommand{\free}[1]{\ensuremath{\mathcal{F}({#1})}}

\newcommand{\IF}{\ensuremath{\mathbf{IF}}}
\newcommand{\iclass}[2]{\ensuremath{\langle{#1}\rangle^{#2}_{\cong}}}
\newcommand{\lint}[4]{\ensuremath{\int_{#1}^{#2}{#3}\:\mathrm{d}{#4}}}
\newcommand{\map}[3]{\ensuremath{{#1}:{#2}\to{#3}}}

\newcommand{\n}[1]{\ensuremath{\left\|{#1}\right\|}}
\newcommand{\ndot}{\ensuremath{\left\|\cdot\right\|}}
\newcommand{\pel}{\ensuremath{\mathbf{P}}}

\newcommand{\pn}[2]{\ensuremath{\left\|{#1}\right\|_{#2}}}

\newcommand{\prsb}{\mathcal{B}}

\newcommand{\Q}{\mathbb{Q}}
\newcommand{\R}{\mathbb{R}}
\newcommand{\restrict}[1]{\ensuremath{\mathord{\upharpoonright}_{#1}}}
\newcommand{\se}{\mathcal{SB}}
\newcommand{\set}[2]{\ensuremath{\left\{{#1}\;:\;\,{#2}\right\}}}
\newcommand{\SIG}[2]{\ensuremath{\mathbf{\Sigma}^{#1}_{#2}}}
\newcommand{\Tr}{\mathbf{Tr}}

\newcommand{\vs}{\mathcal{R}}

\newcommand{\vsd}{\vs_{\mathrm{dis}}}

\newcommand{\WF}{\ensuremath{\mathbf{WF}}}
\newcommand{\Z}{\mathbb{Z}}

\DeclareMathOperator{\aspan}{span}
\DeclareMathOperator{\conv}{conv}
\DeclareMathOperator{\Lip}{Lip}
\DeclareMathOperator{\ran}{ran}

\renewcommand{\geq}{\geqslant}
\renewcommand{\leq}{\leqslant}

\numberwithin{equation}{section}

\allowdisplaybreaks

\renewcommand{\MR}[1]{} 

\begin{document}

\title[The isomorphism class of Pe{\l}czy\'nski's universal basis space]{Lipschitz-free spaces and the isomorphism class of Pe{\l}czy\'nski's universal basis space}

\date{\today}

\author[R. J. Smith]{Richard J. Smith}
\address[R. J. Smith]{School of Mathematics and Statistics, University College Dublin, Belfield, Dublin 4, Ireland}
\email{richard.smith@maths.ucd.ie}

\begin{abstract}
We show that the class of complete separable metric spaces whose Lipschitz-free space is isomorphic to Pe{\l}czy\'nski's universal basis space $\pel$ is $\SIG{1}{1}$-complete.
\end{abstract}
	
\maketitle

\section{Introduction}\label{sect_intro}

Let $M$ be a metric space with base point $0_M$, i.e.~a pointed metric space. The Lipschitz space $\Lip_0(M)$ is the Banach space of real-valued Lipschitz functions on $M$ that vanish at $0_M$, whose norm assigns to each function its optimal Lipschitz constant $\Lip(f)$. The Lipschitz-free (hereafter free) space $\free{M}$ over $M$ (also called the transportation cost space or the Arens-Eells space and denoted $\text{\AE}(M)$) is the closed linear span in $\Lip_0(M)^*$ of the evaluation functionals $\delta(x)$, $x \in M$, defined by $\da{f}{\delta(x)}=f(x)$, $f\in\Lip_0(M)$. Free spaces enjoy close links with metric geometry and optimal transport theory, and have been a popular topic of study by functional analysts for nearly 25 years. For the fundamental properties of these spaces, the reader is encouraged to consult the standard sources \cites{godefroy:kalton:03,weaver:18}.

We present some of the many intriguing characteristics of these spaces. The first is their universal property, which allows us to linearise certain Lipschitz maps:~if $X$ is a Banach space and $\map{f}{M}{X}$ is a Lipschitz map satisfying $f(0_M)=0$, then there exists a unique linear operator $\map{T}{\free{M}}{X}$ satisfying $f=T\circ \delta$ and $\n{T}=\Lip(f)$ \cite{weaver:18}*{Theorem 3.6}. Second is the fact that a Banach space $X$ has the bounded approximation property (BAP) if and only if $\free{X}$ has the BAP \cite{godefroy:kalton:03}*{Theorem 5.3}. Third is the intriguing fact, which follows in part from the second, that $\pel \cong \free{\pel}$, where $\pel$ denotes Pe{\l}czy\'nski's universal basis space (see \cite{pelczynski:69} and \cite{albiac:kalton:16}*{Section 15.3}) and $\cong$ means linear isomorphism \cite{godefroy:kalton:03}*{p.~139}.

Concerning $\pel$ in particular, it follows from \cite{kaufmann:15}*{Corollary 3.3} and the above that $\pel \cong \free{B_\pel}$ (where $B_\pel$ is the closed unit ball of $\pel$), and 
$\pel \cong \free{K}$ for some compact convex set $K\subseteq \pel$ \cite{garcia-lirola:prochazka:18}*{Theorem 1}. 
This leads one to suspect that the class of separable metric spaces $M$ such that $\pel \cong \free{M}$ could be rather diverse. The main result of this paper shows that the class of such $M$ satisfying $\pel \cong \free{M}$ has maximal complexity, relative to the appropriate descriptive set-theoretic framework.

In general, the classification of separable free spaces up to isomorphism is notoriously difficult. Despite several advances made over the last two decades in particular, still relatively little is known. It follows from the universal property that if $M$ and $N$ are Lipschitz isomorphic then $\free{M} \cong \free{N}$, but the converse fails badly (as seen above) and the answers to many easily stated questions, such as whether $\free{\R^2} \cong \free{\R^3}$, still elude us. As discussed further below, our result provides some mathematical justification for the difficulty of classifying these spaces.



The descriptive set-theoretic framework we use will be the same as the one set up in \cite{smith:26}. Denote by $\vs$ the space of all metrics $d$ on $\omega$. This is a $G_\delta$ in $[0,\infty)^{\omega\times \omega}$ equipped with the topology of pointwise convergence, and is thus Polish. This space was first introduced and studied by Vershik in \cites{vershik:98,vershik:04}. To each $d\in\vs$ we associate the completion $M_d:=\cl{(\omega,d)}$ of $(\omega,d)$, with base point $0\in\omega$, together with the free space $\free{M_d}$. A pointed metric space $M$ is complete separable and infinite if and only if $M \equiv M_d$ for some $d\in\vs$, where $\equiv$ denotes isometry, and where the isometry maps $0_M$ to $0$. Thus every separable infinite-dimensional free space $\free{M}$ can be coded by $\vs$, because we know that $M$ in this case must be separable and infinite, and we can assume that it is also complete as well, thus $\free{M} \equiv \free{M_d}$ for some $d\in\vs$ by the above and the universal property (where in the context of Banach spaces $\equiv$ denotes linear isometry).

We must also consider a corresponding framework for infinite-dimensional separable Banach spaces \cite{cuth:doucha:kurka:22a}. Define the Polish space $\prsb$ of functions $\mu:V\to[0,\infty)$ (where $V$ is the $\Q$-linear vector space of all $\Q$-valued elements of $c_{00}$) which can be extended to a norm on $c_{00}$. To each $\mu \in \prsb$ we associate the completion $X_\mu:=\cl{(V,\mu)}$; similarly to above, a Banach space $X$ is separable and infinite-dimensional if and only if $X\equiv X_\mu$ for some $\mu$. This is an alternative to the Effros-Borel space $\se$ introduced by Bossard, which is the established framework for applications of descriptive set theory to Banach spaces \cites{bossard:02,dodos:10}. There is a Borel isomorphism $\map{\Theta}{\prsb}{\se}$ such that $\Theta(\mu) \equiv X_\mu$ for all $\mu \in \prsb$ \cite{cuth:doucha:kurka:22a}*{Theorem 8}. This means that all existing results that use $\se$ can be imported to $\prsb$. For this reason, and because being Polish confers some benefits, we shall use $\prsb$ throughout this paper, substituting it for $\se$ without comment whenever the latter has been used in the past.

Given a separable infinite-dimensional Banach space $X$, we define the isomorphism classes of $X$ with respect to $\prsb$ and $\vs$:
\[
\iclass{X}{} = \set{\mu \in \prsb}{X_\mu \cong X} \quad\text{and}\quad \iclass{X}{\vs} = \set{d \in \vs}{\free{M_d} \cong X},
\]
where $\cong$ denotes linear isomorphism. The former set is well-known to be $\SIG{1}{1}$. The same holds for the latter because there is a continuous map $\map{\Xi}{\vs}{\prsb}$ such that $\free{M_d} \equiv X_{\Xi(d)}$ for all $d \in \vs$ (see Proposition \ref{pr:reduction} below), whence $\iclass{X}{\vs} = \Xi^{-1}(\iclass{X}{})$. Of course, $\iclass{X}{\vs}$ can be empty because $X$ is not obliged to be isomorphic to any free space. This happens, for example, when $X$ lacks a complemented subspace isomorphic to $\ell_1$ \cite{cuth:doucha:wojtaszczyk:16}*{Theorem 1.1}. Hereafter, whenever we refer to $\iclass{X}{\vs}$ we assume it is non-empty.

The descriptive complexity of $\iclass{X}{}$ has been studied comprehensively by several authors. It follows from \cites{bourgain:81,bourgain:rosenthal:schechtman:81}  that $\iclass{L_p[0,1]}{}$, $p \in (1,\infty)\setminus\{2\}$, is not $\DEL{1}{1}$ (the same holds more generally for $\iclass{L_\Phi[0,1]}{}$, where $L_\Phi[0,1]$ is a reflexive Orlicz space not isomorphic to $L_2[0,1]$ \cite{ghawadrah:16}). The sets $\iclass{\pel}{}$ and $\iclass{\ell_2}{}$ are not $\DEL{1}{1}$ and $\DEL{1}{1}$, respectively \cite{bossard:02} (the latter being a consequence of Kwapie\'n’s theorem \cite{kwapien:72}). It was conjectured \cite{bossard:02} that if $\iclass{X}{}$ is $\DEL{1}{1}$ then $X \cong \ell_2$, however this was shown to be false in \cites{godefroy:17a}; shortly thereafter $\iclass{\ell_p}{}$, $p \in (1,\infty)\setminus\{2\}$, were shown to be $\DEL{1}{1}$ \cites{godefroy:17b}. On the other hand, $\iclass{c_0}{}$ is not $\DEL{1}{1}$ \cite{kurka:19}.

The purpose of this paper is start investigating the corresponding question for $\iclass{X}{\vs}$. This was motivated primarily by the cases $X=\ell_1$ or $X=L_1[0,1]$, which are Banach spaces that play key roles in free space theory. If it turns out that $\iclass{X}{\vs}$ is not $\DEL{1}{1}$ in at least one of these cases, then $\iclass{X}{}$ would also not be $\DEL{1}{1}$ as $\iclass{X}{\vs} = \Xi^{-1}(\iclass{X}{})$, thus solving an open problem (see e.g.~\cite{godefroy:17a}*{Problem 3.2}). While we didn't establish the complexity of $\iclass{X}{\vs}$ in these cases, we did do so when $X=\pel$. This brings us to the main result of the paper.


\begin{theorem}\label{th:pel_complete-analytic}
The set $\iclass{\pel}{\vs}$ is $\SIG{1}{1}$-complete.
\end{theorem}

We mention some consequences of this theorem. First, it means that the equivalence relation $\approx$ on $\vs$, defined by $d \approx d'$ if and only if $\free{M_d} \cong \free{M_{d'}}$, is not a $\DEL{1}{1}$ subset of $\vs \times \vs$. In particular, it is not smooth, that is, there is no Borel map $\map{f}{\vs}{\R}$ such that $d \approx d'$ if and only if $f(d)=f(d')$. As mentioned above, this goes towards validating the collective intuition, built up over decades of work,  that free spaces are difficult to classify up to linear isomorphism. Second, we get immediately Bossard's result that $\iclass{\pel}{}$ is not $\DEL{1}{1}$ as $\iclass{\pel}{\vs} = \Xi^{-1}(\iclass{\pel}{})$. Third, recall that, given an equivalence relation $\sim$ on a set $\Gamma$, we call $A \subseteq \Gamma$ a transversal, or section, if the intersection of $A$ with every equivalence class of $\sim$ is a singleton. By following the proof of \cite{bossard:02}*{Proposition 2.7 (i)} almost to the letter, we get that $\approx$ has no $\SIG{1}{1}$ transversal. In particular, we obtain another proof of the now well-established fact that $\approx$ has uncountably many equivalence classes (see \cites{basset:25,basset:lancien:prochazka:25,hajek:lancien:pernecka:16}).

We briefly outline the proof of Theorem \ref{th:pel_complete-analytic}. To each tree $T\in \Tr$ we will associate, in a continuous fashion, a complete infinite separable metric space labelled $M_{\Delta(T)}$, such that $\free{M_{\Delta(T)}} \cong \pel$ whenever $T$ is ill-founded ($T \in \IF$) and $M_{\Delta(T)}$ is discrete otherwise. This means that if $T$ is well-founded ($T\in\WF$) then $\free{M_{\Delta(T)}}$ has the Radon-Nikod\'ym property (RNP) \cite{aliaga:gartand:petitjean:prochazka:22}*{Theorem 4.6} and, as such, $\free{M_{\Delta(T)}} \not\cong\pel$. The template for the construction of the metric spaces $M_{\Delta(T)}$, including the proof that $M_{\Delta(T)}$ is discrete whenever $T\in\WF$, is provided in \cite{smith:26}. What remains is to demonstrate that $\free{M_{\Delta(T)}} \cong \pel$ whenever $T \in \IF$. Using Pe{\l}czy\'nski's decomposition method, this will follow if each space can be shown to be isomorphic to a complemented subspace of the other. Taking advantage of existing results, it will be straightforward to prove that $\pel$ is isomorphic to a complemented subspace of $\free{M_{\Delta(T)}}$ whenever $T \in \IF$. To show that $\free{M_{\Delta(T)}}$ is isomorphic to a complemented subspace of $\pel$, we will establish that $\free{M_{\Delta(T)}}$ has the BAP. Most of the paper is dedicated to proving this last statement.

The rest of the paper is organised as follows. In Section \ref{sec:prelims}, we give the definitions missing from the Introduction and reprise the relevant concepts and results from \cite{smith:26}. In Section \ref{sec:BAP_tree_spaces} we give the proof of Theorem \ref{th:pel_complete-analytic}. To order to overcome a technical obstacle in the proof (more precisely the proof of Lemma \ref{th:tree_BAP}), we introduce the `positive-BAP' for free spaces, which is an apparent strenghening of the usual BAP, and show that the free space over a particular auxiliary metric space has the positive-BAP. Some open questions are given at the end of the paper. We end the Introduction with a remark about Theorem \ref{th:pel_complete-analytic} and Bossard's original proof that $\iclass{\pel}{}$ is not $\DEL{1}{1}$.

\begin{remark}\label{rm:bossard_trivial}
It is worth asking whether Bossard's proof could be adapted to provide a straightforward proof of Theorem \ref{th:pel_complete-analytic} simply by considering the free spaces of the Banach spaces produced in his construction in \cite{bossard:02}*{Theorem 1.1}. The problem with this approach is that it is currently very difficult to distinguish the free spaces of infinite-dimensional separable Banach spaces. Given two such spaces $X$ and $Y$, we know that $\free{X}\not\cong\free{Y}$ if e.g.~$X$ has the BAP and $Y$ does not \cite{godefroy:kalton:03}, or if $X$ is superreflexive and $Y$ is not weakly sequentially complete \cites{aliaga:nous:petitjean:prochazka:21,kochanek:pernecka:18}. These conditions can't be applied to Bossard's original construction (or modified version thereof, as far as we can tell). This is why we have resorted to using complete discrete metric spaces instead of attempting such an adaptation. 
\end{remark}

\section{Notation and preliminaries}\label{sec:prelims}

The notation used for Banach spaces will be standard throughout. Recall that, given $\lambda \geq 1$, a Banach space $X$ is said to have the $\lambda$-bounded approximation property ($\lambda$-BAP) if, given $\ep>0$ and $x_1,\ldots,x_n \in X$, there is a finite-rank linear operator $\map{T}{X}{X}$ such that $\n{T} \leq \lambda$ and $\n{Tx_i-x_i} < \ep$ for $1 \leq i \leq n$. The space $X$ has the BAP if it has the $\lambda$-BAP for some $\lambda\geq 1$. The $1$-BAP is called the metric approximation property or MAP.

Next, we recall some standard facts about free spaces. First, we will exploit the natural identification $\Lip_0(M)\equiv\free{M}^*$ throughout. Second, $\map{\delta}{M}{\free{M}}$ is an isometric embedding and $\delta(M\setminus\{0_M\})$ is a linearly independent set. Third, given $A \subseteq M$ containing $0_M$, the universal property applied to the embedding of $A$ into $M$ yields a linear embedding of $\free{A}$ into $\free{M}$, which is isometric by virtue of McShane's extension theorem; this means we can and do naturally identify $\free{A}$ with a subspace of $\free{M}$. Fourth, let $\map{r}{M}{M}$ be an $L$-Lipschitz retraction onto a set $A\subseteq M$ containing $0_M$. By the universal property, there is a linear operator $\map{P}{\free{M}}{\free{M}}$ satisfying $\n{P}=L$ and $P \circ \delta = \delta \circ r$, which is easily seen to be a projection onto $\free{A}$. This can be used to show that if $M$ is isometric to a $1$-Lipschitz retract of $N$ (where the isometry and retraction don't necessarily preserve base points) then $\free{M}$ is linearly isometric to a 1-complemented subspace of $\free{N}$. Finally, recall that, given distinct $x,y \in M$, the elementary molecule $m_{xy} \in \free{M}$, defined by
\[
 m_{xy} = \frac{\delta(x)-\delta(y)}{d(x,y)},
\]
satisfies $\n{m_{xy}}=1$. By a straightforward Hahn-Banach separation argument we have $B_{\free{M}}=\cl{\conv}\set{m_{xy}}{x \neq y \in M}$. Hence, given a bounded linear operator $\map{T}{\free{M}}{\free{N}}$ between free spaces,
\begin{equation}\label{eqn:operator_norm}
 \n{T} = \sup\set{\n{Tm_{xy}}}{x \neq y \in M}.
\end{equation}

Turning now to notation and concepts from descriptive set theory, for the most part we follow \cite{kechris:95}. The sets of natural numbers and infinite sequences of natural numbers are denoted $\omega$ and $\omega^\omega$, respectively, and natural numbers will be treated as sets in the usual way. We identify the power set of a set $\Gamma$ with $2^\Gamma$. Borel and analytic subsets of Polish spaces are denoted $\DEL{1}{1}$ and $\SIG{1}{1}$, respectively. A subset $A$ of a Polish space $X$ is called $\SIG{1}{1}$-complete if it is $\SIG{1}{1}$ and, whenever $B$ is a $\SIG{1}{1}$ subset of a $0$-dimensional Polish space $Y$, there exists a continuous map $\map{f}{Y}{X}$ such that $B = f^{-1}(A)$ (such a map is called a reduction and we say that $B$ is reducible to $A$).

We denote by $\omega^{<\omega}$ the tree $\bigcup_{n\in\omega} \omega^n$ of all functions having domain some $n \in\omega$ and codomain $\omega$, including the empty function $\varnothing$. Given $s,t \in \omega^{<\omega}$ and $n\in\omega$, we denote by $|s|$ the length of $s$, which is equal to its domain; we write $s \preccurlyeq t$ when $t$ is an extension of $s$, i.e.~$s=t\restrict{|s|}$, and $s \perp t$ when $s,t$ are incomparable; and finally we let $s \wedge t$ denote the greatest element of $\omega^{<\omega}$ less than both $s$ and $t$. A subset $T \subseteq \omega^{<\omega}$ is called a tree if it is downwards-closed, i.e.~$s \in T$ whenever $s \preccurlyeq t$ and $t \in T$. The space of trees, with the compact topology inherited from $2^{\omega^{<\omega}}$, is denoted $\Tr$. The downwards-closure of a set $S \subseteq \omega^{<\omega}$ is the tree $\set{r \in \omega^{<\omega}}{r \preccurlyeq s \text{ for some }s \in S}$. The subsets of $\Tr$ of well-founded and ill-founded trees are labelled $\WF$ and $\IF$, respectively. It is well-known that $\IF$ is $\SIG{1}{1}$-complete and that a $\SIG{1}{1}$ set $A$ is $\SIG{1}{1}$-complete if $\IF$ is reducible to $A$.

Now we introduce the material we need from \cite{smith:26} and prove two additional lemmas.

\begin{proposition}[{\cite{smith:26}*{Proposition 3.1}}]\label{pr:reduction}
There is a continuous map $\map{\Xi}{\vs}{\prsb}$ such that $X_{\Xi(d)}\equiv \free{M_d}$ for all $d\in\vs$.
\end{proposition}

%
%

We denote by $\vsd$ the set $\set{d \in \vs}{M_d \text{ is discrete}}$. We outline a way of manufacturing an abundance of elements of $\vsd$. Fix a metric $\rho$ on $\omega^{<\omega}$ as follows. Let $(e_s)_{s\in \omega^{<\omega}}$ denote the usual basis of $c_{00}(\omega^{<\omega})$, define the injection $\map{\Omega}{\omega^{<\omega}}{c_{00}(\omega^{<\omega})}$ by
\[
 \Omega(s) = \sum_{r \preccurlyeq s} 2^{-|r|}e_r, \quad s \in \omega^{<\omega}.
\]
and then set $\rho(s,t)=\pn{\Omega(s)-\Omega(t)}{\infty}$, $s,t \in \omega^{<\omega}$. It is easy to show that $\rho(s,t) = 2^{-|s \wedge t|-1}$ whenever $s \neq t$ (see \cite{smith:26}*{above Lemma 3.8}). We will make use of this equality numerous times below. 

Next, given a pointed separable metric space $M=(M,\theta)$, fix a dense sequence $(x_i)_{i\in\omega}$ in $M$ whose elements repeat infinitely often. We can and do assume that $x_0=0_M$. Define a second metric $\eta$ on $\omega^{<\omega}$ by
\[
  \eta(s,t) = \rho(s,t)+\theta(x_{|s|},x_{|t|}).
\]
Let $W=\omega^{<\omega} \cup (\omega^{<\omega} \times \{0\})$. We extend $\eta$ to a metric on $W$, again labelled $\eta$, by setting for $x \neq y$
 \[
  \eta(x,y) = \begin{cases}
  1 & \text{if }x,y \notin \omega^{<\omega},\\
  1+\eta(\varnothing,x) & \text{if }x \in \omega^{<\omega},\, y \notin \omega^{<\omega},\\
  1+\eta(\varnothing,y) & \text{if }x \notin \omega^{<\omega},\, y \in \omega^{<\omega}.
  \end{cases}
 \]

Given a tree $T \in \Tr$, we will call the metric space $(T,\eta)$ a tree over $M$ (this should not be confused with the idea of an $\R$-tree, which is something else altogether). The base point of $T$ will always be $\varnothing$. We state the next lemma for use later.

\begin{lemma}[{\cite{smith:26}*{Lemma 3.9}}]\label{lem:combine} Let $M=(M,\theta)$ be a complete separable metric space. There exists a continuous map $\map{\Delta}{\Tr}{\vs}$ such that $\Delta(T) \in \vsd$ whenever $T \in \WF$ and $M$ is isometric to a 1-Lipschitz retract of $M_{\Delta(T)}$ whenever $T \in \IF$.
\end{lemma}

In the proof of this lemma, it is shown that the metric space $(\omega,\Delta(T))$ is isometric to
\[
(T \cup ((\omega^{<\omega}\setminus T) \times \{0\}),\eta).
\]
Having this in mind, we deduce an easily proved yet important fact about $\free{M_{\Delta(T)}}$.

\begin{lemma}\label{lm:tree_BAP}
Let $T \in \Tr$. If $\free{T,\eta}$ has the BAP then so does $\free{M_{\Delta(T)}}$.
\end{lemma}

\begin{proof}
Assume that $T \neq \omega^{<\omega}$, else the conclusion is immediate. Set $W_T= T \cup A_T$, where $A_T=(\omega^{<\omega}\setminus T) \times \{0\}$. It is clear that the map $\map{r}{(W_T,\eta)}{(W_T,\eta)}$, defined by
\[
 r(x) = \begin{cases} x & \text{if }x \in T,\\
  \varnothing & \text{if }x \notin T,
  \end{cases}
\]
is a $1$-Lipschitz retraction onto $T$. By the universal property of free spaces, the unique linear map $\map{P}{\free{W_T,\eta}}{\free{W_T,\eta}}$ satisfying $P\delta(x)=\delta(r(x))$, $x \in W_T$, is a projection onto $\free{T,\eta}$, and it is obvious that $I-P$ maps onto $\free{A_T,\eta}$. As $T \neq \omega^{<\omega}$, $A_T$ is infinite, thus $(A_T,\eta)$ is isometric to $\omega$ equipped with the $0$-$1$ discrete metric. Therefore
\[
 \free{M_{\Delta(T)}} \equiv \free{\omega,\Delta(T)} \equiv \free{W_T,\eta} \cong \free{T,\eta} \oplus \free{A_T,\eta} \cong \free{T,\eta} \oplus \ell_1.
\]
The result follows.
\end{proof}

We end this section with a tool for making natural $1$-$\rho$-Lipschitz retractions from trees onto substrees. 

\begin{lemma}\label{lm:subtree_retraction}
 Let $S$ be a subtree of a tree $T$ and define $\map{v}{T}{S}$ by 
 \[
  v(t) = \max\set{r \in S}{r \preccurlyeq t}.
 \]
 Let $s,t \in T$. Then $s \wedge t \in S$ implies $v(s) \wedge v(t) = s \wedge t$, and $s \wedge t \notin S$ implies $v(s)=v(t)$. Consequently, $v$ is a $1$-$\rho$-Lipschitz retraction onto $S$.
\end{lemma}

\begin{proof}
Assume $s \wedge t \in S$. Clearly $v(s) \wedge v(t) \preccurlyeq s \wedge t$. The fact that $s \wedge t \in S$ implies $s \wedge t \preccurlyeq v(s)$ by maximality of $v(s)$, and likewise $s \wedge t \preccurlyeq v(t)$. Hence $v(s) \wedge v(t) = s \wedge t$. Now assume $s \wedge t \notin S$. Then $v(s), s \wedge t \preccurlyeq s$ implies they are comparable, and $s \wedge t \preccurlyeq v(s)$ is impossible as $S$ is a subtree. Hence $v(s) \prec s \wedge t \preccurlyeq t$, which implies $v(s) \preccurlyeq v(t)$ by maximality of $v(t)$. Similarly $v(t) \preccurlyeq v(s)$, giving equality.

Now assume $v(s) \neq v(t)$. From above, $v(s) \wedge v(t) = s \wedge t$, and thus
\[
 \rho(v(s),v(t)) = 2^{-|v(s) \wedge v(t)|-1} = 2^{-|s \wedge t|-1} = \rho(s,t).
\]
We conclude that $v$ is $1$-Lipschitz with respect to $\rho$.
\end{proof}

Because $v$ does not preserve the heights of elements, it will not be an $\eta$-Lipschitz retraction in general.

\section{The proof of Theorem \ref{th:pel_complete-analytic}}\label{sec:BAP_tree_spaces}

The proof of Theorem \ref{th:pel_complete-analytic} requires a series of lemmas, the most involved being Lemma \ref{th:tree_BAP}. In order to overcome a technical obstacle in the proof of the latter, we will need a certain auxiliary space to have a property apparently stronger than the BAP.

A Banach lattice $X$ is said to have the $\lambda$-positive-BAP if the operators $T$ in the definition of $\lambda$-BAP can be chosen to be positive as well \cite{nielsen:88}. Free spaces enjoy the following weak notion of positivity. An element $m \in \free{M}$ is said to be positive, written $m \geq 0$, if $\da{f}{m} \geq 0$ whenever $f \in \Lip_0(M)$ satisfies $f \geq 0$ \cite{weaver:18}*{p.~96}. While free spaces are almost never Banach lattices with respect to the induced partial order, this idea is nevertheless useful from time to time. We shall call a bounded linear operator $\map{T}{\free{M}}{\free{N}}$ positive if $T\delta(x)\geq 0$ for all $x \in M$. It is obvious that $T$ is positive if and only if $T^*$ is positive, that is, $T^*f \geq 0$ whenever $f \in \Lip_0(N)$ and $f \geq 0$. Thus, the composition of two positive operators is again positive. Having this idea in mind, we will say that $\free{M}$ has the $\lambda$-positive-BAP if the operators $T$ in the definition of $\lambda$-BAP can be chosen to be positive as well. The $1$-positive-BAP will be called the positive-MAP.

Given a pointed metric space $(M,\theta)$, we shall define an extension $(M^*,\theta)$ which also involves a change of base point. Set $M^*=\{0^*\} \cup M$, where $0_{M^*}=0^* \notin M$ is the new base point. Then extend the metric $\theta$ to $M^*$ by setting $\theta(0^*,x)=\theta(x,0^*)=\frac{1}{2}+\theta(0_M,x)$ whenever $x \in M$. The auxiliary space required by Lemma \ref{th:tree_BAP} is $(\frac{1}{2}B_\pel)^*$, where $\frac{1}{2}B_\pel$ has its usual base point $0$ and metric $\theta$ inherited from the canonical norm on $\pel$.

\begin{lemma}\label{lm:M*}
The free space $\free{(\frac{1}{2}B_\pel)^*}$ has the positive-MAP.
\end{lemma}

\begin{proof}
It is sufficient to show that, given $x_1,\ldots,x_n \in \frac{1}{2}B_\pel$ and $\ep>0$, there exists a finite-rank positive linear operator $T$ on $\free{(\frac{1}{2}B_\pel)^*}$ such that $\n{T}\leq 1$ and $\n{T\delta(x_i)-\delta(x_i)}\leq\ep$ for $1 \leq i \leq n$. Given $x_1,\ldots,x_n \in \frac{1}{2}B_\pel$ and $\ep>0$, pick a finite-rank contractive linear projection $q$ on $\pel$ such that $\n{qx_i-x_i} \leq \frac{1}{2}\ep$ for $1 \leq i \leq n$. Since $q\pel$ has dimension $N<\infty$, $q(\frac{1}{2}B_\pel)  \subseteq \frac{1}{2}B_\pel$ is isometric to a compact convex set $K \subseteq \R^N$ (with respect to some norm) having $0$ as an interior point; hereafter we identify $q(\frac{1}{2}B_\pel)$ with $K$. We extend $q$ to a 1-Lipschitz retraction $q^*$ of $(\frac{1}{2}B_\pel)^*$ onto $K^*$ by setting $q^*(0^*)=0^*$ and $q^*(x)=qx$, $x \in \frac{1}{2}B_\pel$. By the universal property, there exists a unique linear map $\map{R}{\free{(\frac{1}{2}B_\pel)^*}}{\free{K^*}}$ such that $R\delta(x)=\delta(q^*(x))$, $x \in (\frac{1}{2}B_\pel)^*$.

According to Lemma \ref{lm:positive_K*} below, there exists a finite-rank positive linear operator $S$ on $\free{K^*}$ such that $\n{S} \leq 1$ and $\n{S\delta(qx_i) - \delta(qx_i)} \leq \frac{1}{2}\ep$ for $1 \leq i \leq n$. Evidently $R$ is also positive. Then $\map{T:=SR}{\free{(\frac{1}{2}B_\pel)^*}}{\free{K^*} \subseteq \free{(\frac{1}{2}B_\pel)^*}}$ is positive, has finite rank and satisfies $\n{T}\leq 1$. Recalling that $\delta$ is an isometric embedding, we obtain
\begin{align*}
\n{T\delta(x_i)-\delta(x_i)} &\leq \n{S\delta(qx_i)-\delta(qx_i)} + \n{\delta(qx_i)-\delta(x_i)}\\ &\leq \tfrac{1}{2}\ep + \n{qx_i-x_i} \leq \ep, \quad 1 \leq i \leq n. \qedhere
\end{align*}
\end{proof}

The proof of the next lemma follows the approach to generating free spaces having the MAP developed in \cite{pernecka:smith:15}. To avoid repetition, we sketch the argument, concentrating on the parts where the arguments differ.

\begin{lemma}[{cf.~\cite{pernecka:smith:15}*{Corollary 1.2}}]\label{lm:positive_K*}
Let $\ndot$ be a norm on $\R^N$ and let $K \subseteq \R^N$ be a compact convex set having $0$ as an interior point (and base point), with metric $\theta$ inherited from $\ndot$. Assume $\n{x} \leq 1$ for all $x \in K$. Then $\free{K^*}$ has the positive-MAP.
\end{lemma}

\begin{proof}[Sketch proof.]
The proof of \cite{pernecka:smith:15}*{Corollary 1.2} requires \cite{pernecka:smith:15}*{Theorem 1.1}, which works by finding, for each $\ep>0$, a dual linear operator $T$ on $\Lip_0(K)$ such that $\n{T} \leq 1+\ep$ and $\n{Tf(x)-f(x)} \leq \ep$ for all $x \in K$ and $f \in B_{\Lip_0(K)}$. The preduals of these operators witness the fact that $\free{K}$ has the MAP. Each $T$ is a composition $T = \Lambda SQ$ of dual operators between Lipschitz spaces over certain metric spaces related to $K$. We need to make minor changes to these operators to ensure that they are positive and apply to Lipschitz spaces over metric spaces corresponding instead to $K^*$.

Fix $L \geq 1$ such that $\frac{1}{L}\n{x} \leq \pn{x}{1},\pn{x}{2} \leq L\n{x}$, $x \in \R^N$. Let $\alpha>\beta>1$, let $B(x,r) \subseteq \R^N$ denote the closed Euclidean ball having centre $x$ and radius $r>0$, and pick $\delta,r>0$ small enough so that $K + B(0,\sqrt{N}\delta) \subseteq \beta K$ and $\beta K + B(0,r) \subseteq \alpha K$ (this is possible by compactness, convexity and the fact that $0$ is an interior point of $K$).

Define $\map{Q}{\Lip_0(K^*)}{\Lip_0((\alpha K)^*)}$ by $Qf(0^*)=0$ and $Qf(x)=f(\frac{x}{\alpha})$ if $x \in \alpha K$. It is immediate that $Q$ is positive, and it is easy to verify that $\n{Q} \leq \frac{1}{\alpha}$ and, using the assumption $\n{x}\leq 1$ whenever $x \in K$,
\[
|Qf(x)-f(x)| \leq \theta(\tfrac{x}{\alpha},x)\n{f} \leq (1-\tfrac{1}{\alpha})\n{f}, \quad x \in K.
\]

On \cite{pernecka:smith:15}*{p.~38}, the operator $S$ is defined via a convolution with a positive smooth function $\eta_r$ supported on $B(0,r)$. Here, we define $\map{S}{\Lip_0((\alpha K)^*)}{\Lip_0((\beta K)^*)}$ by $Sf(0^*)=0$ and
\[
 Sf(x) = \lint{B(0,r)}{}{\eta_r(t)f(x-t)}{t}, \quad x \in \beta K.
\]
Again, it is immediate that $S$ is positive. By adapting elements of the proof of \cite{pernecka:smith:15}*{Lemma 3.2}, we verify that $\n{S} \leq 1+2r$ and $|Sf(x)-f(x)| \leq 2Lr\n{f}$ whenever $x \in \beta K$.

Finally, in \cite{pernecka:smith:15} (specifically equation (12) and pages 42--43), the operator $\Lambda$ is defined via coordinatewise affine interpolation of functions on a lattice of finitely many hypercubes whose vertices belong to a subset of $\delta\Z^N$ and whose union is included in $K+B(0,\sqrt{N}\delta) \subseteq \beta K$. This means that, given $x \in M$, the value of $\Lambda f(x)$ is a finite convex combination of $f(v)$, for some points $v\in\delta\Z^N$. As such, $\Lambda$ is positive and has finite rank. Our operator $\map{\Lambda}{\Lip_0((\beta K)^*)}{\Lip_0(K^*)}$ is defined in the same way, except that we must add $\Lambda f(0^*)=0$ to the definition.

We conclude that the operator $T:=\Lambda SQ$ on $\Lip_0(K^*)$ is positive and has finite rank. Let $\ep>0$. By choosing $\alpha$ close enough to $1$ and $r,\delta$ close enough to $0$, we can ensure that $\n{T} \leq 1+\ep$ and $\theta(Tf(x),f(x))\leq \ep$ whenever $x \in K^*$ and $f \in B_{\Lip_0(K^*)}$. That $\alpha$ and $r$ can be chosen to facilitate this should be clear from above. That $\delta$ can be chosen small enough so that $\n{\Lambda}$ is close enough to $1$, and $\n{\Lambda f(x)-f(x)}$, $x \in K$, is close enough to $0$, is shown in \cite{pernecka:smith:15}*{Lemmas 3.2 and 3.3}.
\end{proof}

The proof of Lemma \ref{th:tree_BAP} also requires us to perturb positive finite-rank operators on free spaces in such a way that the pertubations remain positive. Lemma \ref{lem:positive_perturbation} is the tool we will use to do this. To prove it, we require a result about the integral representation of positive elements of free spaces that follows easily from \cite{aliaga:pernecka:23}*{Corollary 5.8}. 

\begin{lemma}\label{lm:representation}
Let $(M,d)$ be a complete pointed metric space and let $m \in \free{M}$ be positive. Then there is a positive Radon measure $\mu$ on $M\setminus\{0_M\}$ such that $\n{\mu}=\n{m}$ and
\begin{equation}\label{eq:representation}
\da{f}{m} = \lint{M\setminus\{0_M\}}{}{\frac{f(x)}{d(x,0_M)}}{\mu(x)}, \quad f \in \Lip_0(M).
\end{equation}
\end{lemma}

\begin{proof}
According to \cite{aliaga:pernecka:23}*{Corollary 5.8}, there exists a positive almost Radon measure $\nu$ on $M$ such that 
\[
\da{f}{m} = \lint{M}{}{f(x)}{\nu(x)}, \quad f \in \Lip_0(M),
\]
where `almost Radon' means $\nu$ is Borel, $\nu(\{0_M\})=0$ and $\nu\restrict{K}$ is Radon whenever $K \subseteq M$ is compact and does not contain $0_M$. Define $\rho \in B_{\Lip_0(M)}$ by $\rho(x)=d(x,0_M)$, $x \in M$. Then $\n{m}=\da{\rho}{m}$ as $m$ is positive and $f \leq \rho$ whenever $f \in \Lip_0(M)$. Define a positive Borel measure $\mu$ on $M\setminus\{0_M\}$ by setting $\mathrm{d}\mu = \rho\,\mathrm{d}\nu$. This is a finite Radon measure satisfying $\n{\mu}=\n{m}$ because $M$ is complete and
\[
\mu(M\setminus\{0_M\})=\lint{M\setminus\{0_M\}}{}{}{\mu(x)} = \lint{M\setminus\{0_M\}}{}{\rho(x)}{\nu(x)} = \da{\rho}{m} = \n{m}.
\]
Finally, it is clear that \eqref{eq:representation} is satisfied.
\end{proof}

\begin{lemma}\label{lem:positive_perturbation}
Let $(M,d)$ be a complete pointed metric space, let $\map{S}{\free{M}}{\free{M}}$ be a bounded finite-rank positive linear operator, let $E \subseteq M$ be dense with $0_M \in E$, and let $\ep>0$. Then there exists a finite set $F \subseteq E$ containing $0_M$ and a positive linear operator $\map{T}{\free{M}}{\free{F}}$, such that $\n{T-S} < \ep$.
\end{lemma}

\begin{proof}
As $S\free{M}$ is finite-dimensional, there exist distinct $x_1,\ldots,x_n \in M\setminus\{0_M\}$ such that $S\free{M} = \aspan\set{S\delta(x_i)}{1 \leq i \leq n}$. Pick $f_i \in \Lip_0(M)$ such that $\da{f_i}{S\delta(x_j)} = \delta_{i,j}$ and
\[
 Sm = \sum_{i=1}^n \da{f_i}{m}S\delta(x_i), \quad m \in \free{M}.
\]
Fix $L>\sum_{i=1}^n \n{f_i},\max\set{\n{S\delta(x_i)}}{1\leq i \leq n}$. Because $M$ is complete and $S\delta(x_i)$ is positive for $1 \leq i \leq n$, by Lemma \ref{lm:representation} there exist positive Radon measures $\mu_i$ on $M\setminus\{0_M\}$ such that $\n{\mu_i}=\n{S\delta(x_i)}<L$ and
\[
\da{f}{S\delta(x_i)} = \lint{M\setminus\{0_M\}}{}{\frac{f(x)}{d(0_M,x)}}{\mu_i(x)}, \quad f \in \Lip_0(M).
\]
By inner regularity of the $\mu_i$, there exists a compact subset $K\subseteq M\setminus\{0_M\}$ such that $\mu_i(A) < \frac{\ep}{2L}$ for all $i$, where $A:=M\setminus(K \cup \{0_M\})$. Fix $r>0$ such that $d(0_M,x) \geq r$ whenever $x \in K$. It is straightforward to check that, given $f \in B_{\Lip_0(M)}$, the map $x \mapsto f(x)/d(0_M,x)$ is $\frac{2}{r}$-Lipschitz on $K$.

By compactness, there is a partition of $K$ into non-empty Borel subsets $B_1,\ldots,B_k$, all having diameter at most $\frac{r\ep}{8L^2}$. Pick $y_j \in B_j$ and then $z_j \in E$ such that $d(z_j,y_j) < \frac{r\ep}{8L^2}$ for $1\leq j \leq k$. Set $F=\{0_M,z_1,\ldots,z_k\}$, let
\[
 m_i = \sum_{j=1}^k \frac{\mu_i(B_j)}{d(0_M,z_j)}\delta(z_j) \in \free{F}, \quad 1 \leq i \leq n,
\]
and define
\[
 Tm = \sum_{i=1}^n \da{f_i}{m}m_i, \quad m \in \free{M}.
\]

Given $f \in \Lip_0(M)$, having in mind that
\[
\bigg|\frac{f(x)}{d(0_M,x)}-\frac{f(z_j)}{d(0_M,z_j)} \bigg| \leq \frac{2}{r}d(x,z_j) \leq \frac{2}{r}(d(x,y_j)+d(y_j,z_j)) < \frac{\ep}{2L^2}, \quad x \in B_j,
\]
we estimate
\begin{align*}
 |\da{f}{S\delta(x_i)-m_i}| &\leq \bigg|\lint{A}{}{\frac{f(x)}{d(0_M,x)}}{\mu_i(x)} \bigg|+\bigg|\lint{K}{}{\frac{f(x)}{d(0_M,x)}}{\mu_i(x)} - \sum_{j=1}^k \frac{\mu_i(B_j)}{d(0_M,z_j)}f(z_j) \bigg|\\
 &< \frac{\ep}{2L} + \sum_{j=1}^k \lint{B_j}{}{\bigg|\frac{f(x)}{d(0_M,x)} - \frac{f(z_j)}{d(0_M,z_j)}\bigg|}{\mu_i(x)}\\
 &\leq \frac{\ep}{2L} + \sum_{j=1}^k \mu_i(B_j)\frac{\ep}{2L^2} \leq \frac{\ep}{L}.
\end{align*}
It follows that $\n{S\delta(x_i)-m_i} \leq \frac{\ep}{L}$ for $1 \leq i \leq n$ and therefore $\n{T-S} < \ep$.

It remains to verify that $T$ is positive. Let $x \in M$. Because $S\delta(x)$ is positive, for every positive $f \in \Lip_0(M)$,
\begin{align*}
 0 \leq \da{f}{S\delta(x)} = \sum_{i=1}^n f_i(x)\da{f}{S\delta(x_i)} = \lint{M\setminus\{0_M\}}{}{\frac{f(y)}{d(0_M,y)}}{\mu(y)},
\end{align*}
where $\mu := \sum_{i=1}^n f_i(x)\mu_i$. Hence $\mu$ is a positive Radon measure on $M\setminus\{0_M\}$. In particular,
\[
 \sum_{i=1}^n f_i(x)\mu_i(B_j) \geq 0, \quad 1 \leq j \leq k.
\]
Therefore,
\[
 T\delta(x) = \sum_{i=1}^n f_i(x)m_i = \sum_{j=1}^k \frac{1}{d(0_M,z_j)}\bigg(\sum_{i=1}^n f_i(x) \mu_i(B_j)\bigg)\delta(z_j)
\]
is positive.
\end{proof}

The next lemma establishes an estimate involving positive elements of certain finite free spaces. 

\begin{lemma}\label{lem:est}
Let $(M,d)$ be a pointed metric space, let $b,c>0$ be constants, let $A \subseteq M$ be a finite set such that $0_M \in A$, and let $\map{\phi}{A}{\{0_M\} \cup (M\setminus A)}$ be an injection satisfying $\phi(0_M)=0_M$, and $b \leq d(0_M,x)$ and $d(x,\phi(x)) \leq c$ whenever $x \in A\setminus\{0_M\}$. Moreover, let $\map{I}{\free{A}}{\free{M}}$ be the identity embedding and let $\map{J}{\free{A}}{\free{M}}$ be the unique bounded linear map satisfying $J(\delta(x)) = \delta(\phi(x))$, $x \in A$. Then
 \[
  \n{(I-J)m} \leq \frac{c}{b}\n{m}
 \]
whenever $m \in \free{A}$ is positive.
\end{lemma}

\begin{proof}
Let $m\in\free{A}$ be positive. As $A$ is finite, there exist constants $a_i \neq 0$ and distinct $x_i \in A\setminus\{0_M\}$, $1 \leq i \leq n$, such that $m = \sum_{i=1}^n a_i \delta(x_i)$. Because $m$ is positive, the $a_i$ must be positive. By considering again the 1-Lipschitz map $\rho:x \mapsto d(0_M,x)$, we obtain $\sum_{i=1}^n a_i d(0_M,x_i) \leq \n{m}$ (it is not hard to see that we have equality). Therefore
\[
 (I-J)m = \sum_{i=1}^n a_i(\delta(x_i)-\delta(\phi(x_i))) = \sum_{i=1}^n a_id(x_i,\phi(x_i))m_{x_i\phi(x_i)},
\]
giving
\[
 \n{(I-J)m} \leq \sum_{i=1}^n a_i d(x_i,\phi(x_i)) \leq c\sum_{i=1}^n a_i \leq \frac{c}{b}\n{m}. \qedhere
\]
\end{proof}

Now we present our final lemma before the proof of Theorem \ref{th:pel_complete-analytic}, which follows immediately afterwards.

\begin{lemma}\label{th:tree_BAP} The free space $\free{T,\eta}$ has the 2-BAP whenever $T$ is a tree over $\frac{1}{2}B_\pel$.
\end{lemma}

\begin{proof}
Let $T$ be a tree over $\frac{1}{2}B_\pel$. Let $\theta$ be the metric on $\frac{1}{2}B_\pel$ inherited from the canonical norm on $\pel$. Recall that we require a dense sequence $(x_i)_{i\in\omega}$ in $\frac{1}{2}B_\pel$ whose elements repeat infinitely often and which satisfies $x_0=0$. Let $S \subseteq T$ be a finite set and let $n \in \omega$. We will construct a finite-rank operator $\map{K}{\free{T}}{\free{T}}$ such that $\n{K} \leq 2(1+2^{-n})^2$ and $K\delta(s) = \delta(s)$ whenever $s \in S$.
 
By taking the downwards-closure of $S$ if necessary, we shall assume, without loss of generality, that $S$ is a finite subtree of $T$. Let $\map{v}{T}{S}$ be the map from Lemma \ref{lm:subtree_retraction}. Of course, the fibres $v^{-1}(r)$, $r \in S$, partition $T$. Let the subtree $T_r$, $r \in S$, be the downwards-closure of $v^{-1}(r)$, which is the union of $v^{-1}(r)$ and the set of predecessors of $r$. It will be useful to have in mind that $t \in T_r$ and $|t| \geq |r|$ implies $v(t)=r$. Define
\[
 I = \set{r \in S}{\sup\set{|t|}{t \in T_r} < \infty}.
\]
In order to define $K$, we must set up some linear maps, two of which will depend on $r \in S$ and whether $r \in I$ or not.

First, assume that $r \in I$. As $\sup\set{|t|}{t \in T_r} < \infty$, there exists $u_r \in T_r$, which we fix from now on, such that $|t| \leq |u_r|$ for all $t \in T_r$. Define $\map{z_r}{T_r}{T_r}$ by $z_r(t)=u_r\restrict{|t|}$. The map $z_r$ is a $1$-$\eta$-Lipschitz retraction onto the set of predecessors of $u_r$. Indeed, for $s \neq t$,
\begin{align*}
 \eta(u_r\restrict{|s|},u_r\restrict{|t|}) &\leq 2^{-\min\{|s|,|t|\}-1} + \theta(x_{|s|},x_{|t|})\\
 &\leq 2^{-|s \wedge t|-1} + \theta(x_{|s|},x_{|t|}) = \eta(s,t).
\end{align*}
By the universal property, there exists a linear map $\map{Z_r}{\free{T_r}}{\free{T_r}}$ satisfying $\n{Z_r} = 1$ and $Z_r\delta(t)=\delta(z_r(t))$ whenever $t \in T_r$.

We will need three linear maps to take care of any points $r \in S\setminus I$. The first two are independent of $r$. Define $\map{w}{(T,\eta)}{((\frac{1}{2}B_\pel)^*,\theta)}$ by $w(\varnothing)=0^*$ and $w(t)=x_{|t|}$ if $t \neq \varnothing$. Let $s,t \in T\setminus\{\varnothing\}$. Clearly $\theta(w(s),w(t)) \leq \eta(s,t)$, and
\[
 \theta(w(\varnothing),w(t)) = \theta(0^*,x_{|t|}) = \tfrac{1}{2}+\theta(0,x_{|t|}) = 2^{-1}+\theta(x_0,x_{|t|})= \eta(\varnothing,t).
\]
Therefore $w$ is a $1$-Lipschitz map that preserves the base points, so again by the universal property there exists a linear map $\map{W}{\free{T}}{\free{(\frac{1}{2}B_\pel)^*}}$ satisfying $\n{W} \leq 1$ and $W\delta(t)=\delta(w(t))$ whenever $t \in T$.

Next, as $S$ is finite, we can assume, without loss of generality, that $n > |s|$ whenever $s \in S$ and $n > |u_r|$ whenever $r \in I$. By Lemmas \ref{lm:M*} and \ref{lem:positive_perturbation}, there exists $p \geq n$ and a positive finite-rank operator $\map{G}{\free{(\frac{1}{2}B_\pel)^*}}{\free{E^*}}$, where $E=\set{x_i}{i<p}$, such that $\n{G} \leq 1+2^{-n}$ and $\n{\delta(x_i) - G\delta(x_i)}\leq 2^{-2n}$ whenever $i<n$.

Now we define our fourth type of linear map. Given $i,j < p$, let $i \sim j$ if $x_i = x_j$, and let $F \subseteq p$ be the set of least elements of all equivalence classes of $\sim$. Set
\[
\ep = \min\set{d(x_i,x_j)}{i,j \in F,\, i \neq j}>0,
\]
and let $\map{\pi}{F}{\omega}$ be a map such that $n \leq \pi(i)$, $2^{-\pi(i)-1} \leq 2^{-n}\ep$ and $x_{\pi(i)}=x_i$ whenever $i \in F$. The existence of such a map, which is necessarily an injection, follows from the assumption that the $x_i$ repeat infinitely often.

Fix $q=\max\ran\pi + 1$. Given $r \in S\setminus I$, because $\sup\set{|t|}{t \in T_r}=\infty$, there exists $u_r \in T_r$ such that $|u_r|=q$. Define $\map{h_r}{(E^*,\theta)}{(T_r,\eta)}$ by
\[
h_r(x) = \begin{cases}
\varnothing & \text{if }x=0^*,\\
u_r\restrict{\pi(i)} & \text{if $x=x_i$ for some (unique) $i \in F$}.
\end{cases}
\]
We verify that $\Lip(h_r) \leq 1+2^{-n}$. Indeed, given distinct $i,j \in F$,
\begin{align*}
\eta(h_r(x_i),h_r(x_j))=\eta(u_r\restrict{\pi(i)},u_r\restrict{\pi(j)}) &= 2^{-\min\{\pi(i),\pi(j)\}-1} + \theta(x_{\pi(i)},x_{\pi(j)})\\
&\leq 2^{-n}\ep + \theta(x_i,x_j) \leq (1+2^{-n})\theta(x_i,x_j),
\end{align*}
and
\[
 \eta(h_r(0^*),h_r(x_i))=\eta(\varnothing,u_r\restrict{\pi(i)}) = 2^{-1}+\theta(x_0,x_{\pi(i)}) = \theta(0^*,x_i).
\]

Once more by the universal property, there exists a linear map $\map{H_r}{\free{E^*}}{\free{T_r}}$ satisfying $\n{H_r} \leq 1+2^{-n}$ and $H_r\delta(x_i)=\delta(h_r(x_i))=\delta(u_r\restrict{\pi(i)})$ whenever $i \in F$.

Now define a finite-rank operator $\map{K}{\free{T}}{\free{T}}$ by first setting
\[
K\delta(t) = \begin{cases}
\delta(t) & \text{if }t \in S,\\
Z_r\delta(t) & \text{if }t \in T_r\setminus S \text{ and } r \in I,\\
H_rGW\delta(t) & \text{if }t \in T_r\setminus S \text{ and } r \in S\setminus I,
\end{cases}
\]
and then extending to $\free{T}$ by linearity (recall that the $\delta(t)$, $t \neq \varnothing$, are linearly independent) and continuity (eventually, once \eqref{eqn:inequality_we_need} below has been established). To show that $\n{K} \leq 2(1+2^{-n})^2$, by \eqref{eqn:operator_norm} it suffices to prove
\begin{equation}\label{eqn:inequality_we_need}
\n{K\delta(s)-K\delta(t)} \leq 2(1+2^{-n})^2\eta(s,t),
\end{equation}
whenever $s,t \in T$. Establishing \eqref{eqn:inequality_we_need} will involve a number of cases.

First we assume $s \in S$. If $t \in S$ as well then obviously \eqref{eqn:inequality_we_need} holds. Instead, let $t \in T_r\setminus S$ for some $r \in S$. Then $|t| > |r|$ (else $t \prec r$ and $t \in S$ as $S$ is a subtree) so $v(t)=r$. By Lemma \ref{lm:subtree_retraction}, $s \wedge t \in S$ implies $s \wedge t = v(s) \wedge v(t) = s \wedge r$. We need to distinguish the cases $r \in I$ and $r \in S\setminus I$.

Assume $r \in I$. Then, because $z_r(t) \in T_r$ and $|z_t(t)|=|t| > |r|$, as above we get $v(z_r(t))=r$ and $s \wedge z_r(t) = v(s) \wedge v(z_r(t)) = s \wedge r = s \wedge t$. Using this fact we obtain
\begin{align*}
 \n{K\delta(s)-K\delta(t)} = \n{\delta(s)-\delta(z_r(t))} &= \eta(s,z_r(t))\\
 &= 2^{-|s\wedge z_r(t)|-1} + \theta(x_{|s|},x_{|z_r(t)|})\\
 &= 2^{-|s\wedge t|-1} + \theta(x_{|s|},x_{|t|}) = \eta(s,t).
\end{align*}

Assume instead that $r \in S\setminus I$. First we dispense with the case $s=\varnothing$. In this scenario,
\[
 \n{K\delta(s)-K\delta(t)} = \n{K\delta(t)} \leq (1+2^{-n})^2\n{\delta(t)} = (1+2^{-n})^2\eta(s,t).
\]
Now assume $s \in S\setminus\{\varnothing\}$, so that $W\delta(s)=\delta(w(s)) = \delta(x_{|s|})$. As $|s| < n$ by assumption and $n \leq p$, there exists $i \in F$ such that $x_{|s|}=x_i$, thus
\begin{equation}\label{H_r_equality}
H_r\delta(x_{|s|}) = H_r\delta(x_i) = \delta(h_r(x_i)) = \delta(u_r\restrict{\pi(i)}).
\end{equation}
Because $h_r(x_i) \in T_r\setminus S$, as above, $s \wedge h_r(x_i) = s \wedge r = s \wedge t$ and hence
\begin{equation}\label{metric_r_est_2}
\eta(s,h_r(x_i)) = 2^{-|s \wedge h_r(x_i)|-1} + \theta(x_{|s|},x_{\pi(i)}) = 2^{-|s \wedge t|-1}.
\end{equation}

Using \eqref{H_r_equality}, \eqref{metric_r_est_2} and the fact that $K\delta(t)=H_rG\delta(x_{|t|})$, we can estimate
\begin{align}
& \n{K\delta(s)-K\delta(t)} \nonumber\\
=\;& \n{\delta(s)- H_rG\delta(x_{|t|})} \nonumber\\
\leq\;& \n{\delta(s)-H_r\delta(x_{|s|})} + \n{H_r\delta(x_{|s|})-H_rG\delta(x_{|s|})} + \n{H_rG\delta(x_{|s|})-H_rG\delta(x_{|t|})} \nonumber\\
\leq\;& \eta(s,h_r(x_i)) + \n{H_r}\|\delta(x_{|s|})-G\delta(x_{|s|})\| + \n{H_r}\n{G}\theta(x_{|s|},x_{|t|}) \nonumber\\
\leq\;& 2^{-|s \wedge t|-1} + (1+2^{-n})2^{-2n}+(1+2^{-n})^2\theta(x_{|s|},x_{|t|}) \nonumber\\
\leq\;& (1+2^{-n})^2(2^{-|s \wedge t|-1}+\theta(x_{|s|},x_{|t|})) \nonumber\\
=\;& (1+2^{-n})^2\eta(s,t). \label{eqn:long}
\end{align}
This completes the case when $s \in S$. 

Hereafter assume $s \in T_r\setminus S$ and $t \in T_{r'}\setminus S$ for some $r,r' \in S$. As above, $|s|>|r|$, giving $v(s)=r$; likewise $v(t)=r'$.  First suppose $r=r'$. If $r \in I$ then
\[
 \n{K\delta(s)-K\delta(t)} = \n{Z_r(\delta(s)-\delta(t))} \leq \eta(s,t),
\]
and if $r \in S\setminus I$ then
\[
 \n{K\delta(s)-K\delta(t)} = \n{H_rG(\delta(x_{|s|})-\delta(x_{|t|}))} \leq (1+2^{-n})^2\theta(x_{|s|},x_{|t|})
 \leq (1+2^{-n})^2\eta(s,t).
\]

Now we assume $r \neq r'$. By Lemma \ref{lm:subtree_retraction}, $v(s)=r \neq r'=v(t)$ implies $s \wedge t \in S$, giving $s \wedge t=r \wedge r'$. We have three final cases:~$r,r' \in I$, $r \in I$, $r' \notin I$ (or vice-versa) and $r,r' \notin I$. First, suppose that $r,r' \in I$. Arguing as above, we have $v(z_r(s))=r$ and $v(z_{r'}(t))=r'$, and
\[
z_r(s) \wedge z_{r'}(t) = v(z_r(s)) \wedge v(z_{r'}(t)) = v(s) \wedge v(t) = s \wedge t.
\]
Hence
\begin{align*}
 \n{K\delta(s)-K\delta(t)} = \n{Z_r\delta(s)-Z_{r'}\delta(t)} &= \eta(z_r(s),z_{r'}(t))\\
 &= 2^{-|z_r(s) \wedge z_{r'}(t)|-1} + \theta(x_{|z_r(s)|},x_{|z_{r'}(t)|})\\
 &= 2^{-|s\wedge t|-1} + \theta(x_{|s|},x_{|t|}) = \eta(s,t).
\end{align*}
Second, suppose $r \in I$ and $r' \in S\setminus I$. As $s \in T_r$ and $r \in I$, we have $|s| \leq |u_r| < n$. Hence, as above, there exists $i \in F$ such that $x_{|s|}=x_i$, thus
\[
H_{r'}\delta(x_{|s|}) = H_{r'}\delta(x_i) = \delta(h_{r'}(x_i)) = \delta(u_{r'}\restrict{\pi(i)}).
\]
Moreover, $z_r(s) \in T_r$ and $|z_r(s)|=|s|>|r|$, and $h_{r'}(x_i) \in T_{r'}$ and $|h_{r'}(x_i)| = \pi(i) \geq n > |r'|$, so $z_r(s) \wedge h_{r'}(x_i) = s \wedge t$ as above. We calculate
\begin{align*}
 \eta(z_r(s),h_{r'}(x_i)) &= 2^{-|z_r(s) \wedge h_r(x_i)|-1} + \theta(x_{|z_r(s)|},x_{\pi(i)})\\
 &= 2^{-|s\wedge t|-1} + \theta(x_{|s|},x_{\pi(i)}) = 2^{-|s\wedge t|-1}.
\end{align*}
This means that we can repeat the estimate \eqref{eqn:long} above with $\delta(s)$ and $H_r$ replaced by $\delta(z_r(s))$ and $H_{r'}$, respectively, and obtain the same outcome.

Finally we consider the case $r,r' \in S\setminus I$. This is the obstacle that prompts the material above on the positive-BAP; we require $\free{(\frac{1}{2}B_\pel)^*}$ to have the positive-MAP and not just the MAP, so that $G$ is positive and so that we can apply Lemma \ref{lem:est}. Once again, for $i\in F$, we have
\[
 h_r(x_i) \wedge h_{r'}(x_i) = v(h_r(x_i)) \wedge v(h_{r'}(x_i)) = v(s) \wedge v(t) = r \wedge r' = s \wedge t.
\]
Let $A=h_r(E^*)$ and define $\map{\phi}{A}{T}$ by $\phi(h_r(x))=h_{r'}(x)$, $x \in E^*$. Given $i \in F$,
\[
 \eta(\varnothing,h_r(x_i)) = \tfrac{1}{2} + \theta(x_0,x_{\pi(i)}) \geq \tfrac{1}{2}
\]
and
\[
 \eta(h_r(x_i),\phi(h_r(x_i))) = 2^{-|h_r(x_i) \wedge h_{r'}(x_i)|}+\theta(x_{\pi(i)},x_{\pi(i)}) = 2^{-|r \wedge r'|-1}=2^{-|s \wedge t|-1}.
\]
Let $\map{J}{\free{A}}{\free{T}}$ be the unique linear map satisfying
\[
J\delta(h_r(x))=\delta(\phi(h_r(x)))=\delta(h_{r'}(x)), \quad x \in E^*.
\]
By Lemma \ref{lem:est}, 
\begin{equation}\label{eqn:est}
\n{(I-J)m} \leq 2^{-|s \wedge t|}\n{m},
\end{equation}
whenever $m \in \free{A}$ is positive. We observe that $H_{r'}=JH_r$ because
\[
 H_{r'}\delta(x) = \delta(h_{r'}(x)) = J\delta(h_r(x)) = JH_r\delta(x), \quad x \in E^*.
\]
Therefore, if we appeal to \eqref{eqn:est} with $m=H_rGW\delta(s) \in \free{A}$, which is positive because $G$ is positive, and $H_r$ and $W$ are positive by inspection, while noting also that
\[
\n{\delta(s)} = \eta(\varnothing,s) = \tfrac{1}{2}+\theta(x_0,x_{|s|}) = \tfrac{1}{2}+\n{x_{|s|}} \leq 1,
\]
we obtain
\[
 \n{(I-J)H_rGW\delta(s)} \leq 2^{-|s \wedge t|}\n{m} \leq (1+2^{-n})^2 2^{-|s \wedge t|}.
\]
Consequently,
\begin{align*}
 \n{K\delta(s) - K\delta(t)} &\leq \n{(H_r - H_{r'})GW\delta(s)} + \n{H_{r'}GW(\delta(s) - \delta(t))}\\
 &\leq \n{(I-J)H_rGW\delta(s)} + (1+2^{-n})^2\theta(x_{|s|},x_{|t|})\\
 &\leq (1+2^{-n})^2(2^{-|s \wedge t|} + \theta(x_{|s|},x_{|t|}))\\
 &\leq 2(1+2^{-n})^2\eta(s,t).
\end{align*}
Thus \eqref{eqn:inequality_we_need} is established and the proof is complete.
\end{proof}

\begin{proof}[Proof of Theorem \ref{th:pel_complete-analytic}]
We fix $M=\frac{1}{2}B_{\pel}$ with the metric $\theta$ induced by the canonical norm on $\pel$. According to \cite{godefroy:kalton:03}*{p.~139} and \cite{kaufmann:15}*{Corollary 3.3}, $\pel \cong \free{\pel} \cong \free{M}$. First assume $T\in \IF$. By Lemma \ref{lem:combine}, $M$ is isometric to a 1-Lipschitz retract of $M_{\Delta(T)}$, and thus $\pel$ is isomorphic to a complemented subspace of $\free{M_{\Delta(T)}}$. By Lemmas \ref{th:tree_BAP} and \ref{lm:tree_BAP}, $\free{M_{\Delta(T)}}$ has the BAP. This means that $\free{M_{\Delta(T)}}$ is isomorphic to a complemented subspace of $\pel$ (see e.g.~\cite{pelczynski:71}). As $\pel \cong \bigoplus_{\ell_2}\pel$ (e.g.~see the proof of \cite{albiac:kalton:16}*{Theorem 15.3.1}), by a standard application of Pe{\l}czy\'nski's decomposition method we get $\free{M_{\Delta(T)}} \cong \pel$. Now suppose $T\in \WF$. Then $M_{\Delta(T)}$ is discrete, again by Lemma \ref{lem:combine}, hence $\free{M_{\Delta(T)}}$ has the RNP \cite{aliaga:gartand:petitjean:prochazka:22}*{Theorem 4.6} whenever $T \in \WF$, and thus $\free{M_{\Delta(T)}} \not\cong \pel$. Therefore $\Delta$ witnesses the fact that $\IF$ is reducible to $\iclass{\pel}{\vs}$, meaning that the latter is $\SIG{1}{1}$-complete.
\end{proof}

We finish the paper with some open problems prompted by remarks in the Introduction. 

\begin{problem}
 Are $\iclass{\ell_1}{\vs}$ or $\iclass{L_1[0,1]}{\vs}$ $\DEL{1}{1}$ or not? More generally, does there exist $X$ such that $\iclass{X}{\vs}$ is non-empty and $\DEL{1}{1}$?
\end{problem}

Concerning the fact that the relation $\approx$ on $\vs$ is not smooth, and having in mind \cite{bossard:02}*{Proposition 6.1}, we can pose the following problem.

\begin{problem}\label{pb:E_0}
 Let $E_0$ be the equivalence relation on $2^\omega$ defined by $x E_0 y$ if and only if $x(i) \neq y(i)$ for finitely many $i\in\omega$. Does there exist a Borel map $\map{f}{2^\omega}{\vs}$ such that $x E_0 y$ if and only if $f(x) \approx f(y)$?
\end{problem}

We remark that, as $E_0$ has continuum many equivalance classes, a positive answer to Problem \ref{pb:E_0} would imply that there is a family of separable metric spaces of size continuum whose free spaces are pairwise non-isomorphic. Showing the existence of a suitable family of this size is probably a necessary first step in any positive answer to Problem \ref{pb:E_0}. Currently, all we know (in ZFC) is that there are families of this kind that are uncountable.

\begin{problem}\label{pb:c}
Is there a family of separable metric spaces of size continuum whose free spaces are pairwise non-isomorphic?
\end{problem}

\section*{Acknowledgements} I thank R.~Aliaga, M.~C\'uth, G.~Godefroy and E.~Perneck\'a, whose remarks helped me during the preparation of this paper.

\bibliography{pelczynski}

\end{document}